\documentclass[10pt]{amsart}
\usepackage{amssymb,amsbsy,amsmath,amsfonts,amssymb,amscd}
\usepackage{latexsym,euscript,exscale}
\title[Generic measure preserving homeomorphisms and isometries]{The generic isometry and measure preserving homeomorphism are conjugate to their powers}
\author {Christian Rosendal}
\address{Department of Mathematics, Statistics, and Computer Science (M/C 249)\\
University of Illinois at Chicago\\
851 S. Morgan St.\\
Chicago, IL 60607-7045\\
USA}
\email{rosendal@math.uic.edu}
\urladdr{http://www.math.uic.edu/$_~$rosendal}
\newcommand {\ca} {{2^\N}}

\newcommand {\A}{{\mathbf A}}
\newcommand {\B}{{\mathbf B}}
\newcommand {\D}{{\mathbf D}}
\newcommand{\F}{\mathbb F}

\newcommand{\aaa}{\go A}
\newcommand {\N}{\mathbb N}
\newcommand {\Q}{\mathbb Q}
\newcommand {\R}{\mathbb R}
\newcommand {\Z}{\mathbb Z}
\newcommand {\C}{{\mathbf C}}
\newcommand {\U}{\mathbb U}

\renewcommand{\leq}{\ensuremath{\leqslant}}
\renewcommand{\geq}{\ensuremath{\geqslant}}

\newcommand {\E}{\mathbf E}

\newcommand{\eps}{\epsilon}

\newcommand{\iso}{\cong}

\newcommand{\tom} {\emptyset}

\newcommand{\inj}{\hookrightarrow}

\newcommand{\equi}{\Leftrightarrow}

\newcommand{\til}{\rightarrow}

\newcommand{\Lim}[1]{\mathop{\longrightarrow}\limits_{#1}}

\newcommand {\del}{ \; \big| \;}

\newcommand {\og}{\; \& \;}

\newcommand {\go} {\mathfrak}
\newcommand {\ku} {\mathcal}

\newcommand{\inv}{^{-1}}

\newcommand {\e} {\exists}
\renewcommand {\a} {\forall}

\newtheorem{thm}{Theorem}
\newtheorem{cor}[thm]{Corollary}
\newtheorem{lemme}[thm]{Lemma}
\newtheorem{prop} [thm] {Proposition}

\begin{document}
\thanks{Research partially supported by NSF grant DMS 0556368.}


\begin{abstract}It is known  that there is a
comeagre set of mutually conjugate measure preserving homeomorphisms
of Cantor space equipped with the coinflipping probability measure, i.e., Haar measure.
We show that the generic measure preserving homeomorphism is
moreover conjugate to all of its powers. It follows that the generic
measure preserving homeomorphism extends to an action of $(\Q,+)$ by
measure preserving homeomorphisms, and, in fact, to an action of the locally compact ring $\aaa$ of finite ad\`eles.

Similarly, S. Solecki has proved that there is a comeagre set of mutually conjugate isometries of the rational Urysohn metric space. We prove that these are all conjugate with their powers and therefore also embed into  $\Q$-actions. In fact, we extend these actions to actions of $\aaa$ as in the case of measure preserving homeomorphisms.

We also consider a notion of topological similarity in Polish groups and use this to give simplified proofs of the meagreness of conjugacy classes in the automorphism group of the standard probability space and in the isometry group of the Urysohn metric space.
\end{abstract}

\maketitle

\tableofcontents

\section{Introduction}

Suppose $M$ is a compact metric space and let ${\rm Homeo}(M)$ be its group of homeomorphisms. We equip ${\rm Homeo}(M)$ with the topology of uniform convergence or what is equivalent, since $M$ is compact metric, the compact-open topology. Thus, in this way, a neighbourhood basis at the identity consists of the sets
$$
\{h\in {\rm Homeo}(M)\del h(C_1)\subseteq V_1\;\&\;\ldots\;\&\; h(C_n)\subseteq V_n\},
$$
where $V_i\subseteq M$ are open and $C_i\subseteq V_i$ compact. Under this topology the group operations are continuous and thus ${\rm Homeo}(M)$ is a topological group. Moreover, the topology is {\em Polish}, that is, ${\rm Homeo}(M)$ is separable and its topology can be induced by a complete metric.

Now consider the case when $M$ is Cantor space $\ca$. Then, as any two disjoint closed sets in $\ca$ can be separated by a clopen set, we get a neighbourhood basis at the identity consisting of sets of the form
$$
\{h\in {\rm Homeo}(\ca)\del h(C_1)=C_1\;\&\;\ldots\;\&\; h(C_n)=C_n\},
$$
where $C_1,\ldots,C_n\subseteq \ca$ is a partition of $\ca$ into clopen sets.

By Stone duality, the homeomorphisms of Cantor space are just the automorphisms of the Boolean algebra of clopen subsets of $\ca$, which we denote by $\B_\infty$. Thus, viewed in this way, the neighbourhood basis at the identity has the form
$$
\{h\in {\rm Homeo}(\ca)\del h|_\C={\rm id}_\C\}
$$
where $\C$ is a finite subalgebra of $\B_\infty$.

Cantor space $\ca$ is of course naturally homeomorphic to the Cantor group $(\Z_2)^\N$ and therefore comes equipped with Haar measure $\mu$. Up to a homeomorphism of Cantor space, $\mu$ is the unique atomless Borel probability measure on $\ca$ such that
\begin{itemize}
\item if $C\in \B_\infty$, then $\mu(C)$ is a {\em dyadic rational}, i.e., of the form $\frac n{2^k}$,
\item if $C\in \B_\infty$ and $\mu(C)=\frac n{2^k}>0$, then for all $l\geq k$, there is some clopen $B\subseteq C$ such that $\mu(B)=\frac 1{2^l}$,
\item if $\tom\neq C\in \B_\infty$, then $\mu (C)> 0$.
\end{itemize}
The measure $\mu$ is of course the product probability measure of the coinflipping measure on each factor $2=\{0,1\}$. For simplicity, we call $\mu$ Haar measure on $\ca$.

One easily sees that the group of Haar measure preserving homeomorphisms ${\rm Homeo}(\ca,\mu)$ of $\ca$ is a closed subgroup of ${\rm Homeo}(\ca)$ and therefore a Polish group in its own right.
It was proved by A. S. Kechris and C. Rosendal in \cite{turbulence} that there are comeagre conjugacy classes in both ${\rm Homeo}(\ca)$ and ${\rm Homeo}(\ca,\mu)$. In fact, the result for ${\rm Homeo}(\ca,\mu)$ is rather simple and also holds for many other sufficiently homogeneous measures on $\ca$ (see E. Akin \cite{akin}). This result allows us to refer to the {\em generic} measure preserving homeomorphism of Cantor space (with Haar measure), knowing that generically they are all mutually conjugate. One of the aims of this paper is to show that the generic measure preserving homeomorphism is conjugate to its non-zero powers, which in turn will show that it is a part of an action of the additive group $(\Q, +)$ by measure preserving homeomorphisms of $\ca$.

In one sense this is an optimal result, as we cannot extend these actions of $(\Q,+)$ to actions of $(\R,+)$. For, as ${\rm Homeo}(\ca,\mu)$ is totally disconnected, there are no non-trivial continuous homomorphism (or even measurable homomorphisms) from $\R$ into ${\rm Homeo}(\ca,\mu)$, and thus $\R$ cannot act non-trivially by (measure preserving) homeomorphisms on $\ca$.
However, we shall see that the generic measure preserving homeomorphism generates a closed subgroup of ${\rm Homeo}(\ca,\mu)$, that is topologically isomorphic to the profinite completion of the integers, and this allows us to extend this group within ${\rm Homeo}(\ca,\mu)$ to the additive group $(\aaa,+)$ of the locally compact ring $\aaa$ of finite ad\`eles by carefully adding roots.

Our result is the natural analogue of a result due to T. de la Rue and J. de Sam Lazaro \cite{rue} stating that for the generic element $g\in {\rm Aut}([0,1],\lambda)$ there is a continuous homomorphism $\phi\colon \R\til {\rm Aut}([0,1],\lambda)$ such that $\phi(1)=g$, i.e., that the generic measure preserving transformation is in the image of a $1$-parameter subgroup. Of course, our group ${\rm Homeo}(\ca,\mu)$ sits inside ${\rm Aut}([0,1],\lambda)$ as a dense subgroup, but the topology on ${\rm Homeo}(\ca,\mu)$ is much finer than that induced from ${\rm Aut}([0,1],\lambda)$, and there seems to be no way of directly relating the two results. 

Our result also gives hope that one could develop some rudimentary ad\`elic Lie theory in ${\rm Homeo}(\ca,\mu)$, since our result implies that there is a rich supply of $1$-parameter ad\`elic subgroups of ${\rm Homeo}(\ca,\mu)$. There have been many attempts of expanding Lie theory to a more general context of topological groups, e.g., W. Wojty\'nski \cite{ww}, but there are also hindrances to this for the groups treated in this paper. For example, almost all of the non-trivial properties developed in \cite{ww} depend on the topological group being {\em analytic}, i.e., that the intersection of the closed central descending sequence is trivial. In our case, however, every element of ${\rm Homeo}(\ca,\mu)$ is a commutator and so all terms of the central descending sequence are just ${\rm Homeo}(\ca,\mu)$ itself. Nevertheless, it would be interesting to see if alternative developments are possible. This would certainly also provide a strong external motivation for expanding the ideas presented here.

\

The {\em Urysohn metric space} $\U$ is a universal separable metric space first constructed by P. Urysohn in the posthomously published
\cite{ury}. It soon went out of fashion following the discovery that many separable Banach spaces are already universal separable metric spaces, but has come to the forefront over the last twenty years as an analogue of  Fra\"iss\'e theory in the case of metric spaces.

The Urysohn space $\U$ is characterised up to isometry by being separable and complete, together with the  following extension property.
\begin{quotation}
If $\phi\colon\A\til \U$ is an isometric embedding of a
finite metric space $\A$ into $\U$ and $\B=\A\cap \{y\}$ is a one
point metric extension of $\A$, then $\phi$ extends to an isometric
embedding of $\B$ into $\U$.
\end{quotation}
There is also a rational variant of $\U$ called the {\em rational Urysohn metric space}, which we denote by $\Q\U$. This is, up to isometry, the unique countable metric space with only rational distances such that the following variant of the above extension property holds.
\begin{quotation}
If $\phi\colon\A\til \Q\U$ is an isometric embedding of a
finite  metric space $\A$ into $\Q\U$ and $\B=\A\cap \{y\}$ is a one
point metric extension of $\A$ whose metric only takes rational distances, then $\phi$ extends to an isometric
embedding of $\B$ into $\Q\U$.
\end{quotation}
We denote by ${\rm Iso}(\Q\U)$ and ${\rm Iso}(\U)$ the isometry groups of $\Q\U$ and $\U$ respectively. These are Polish groups when equipped with the topology of pointwise convergence on $\Q\U$ seen as a discrete set and $\U$ seen as a metric space respectively. Thus, the basic neighbourhoods of the identity in ${\rm Iso}(\Q\U)$ are of the form
$$
\{h\in {\rm Iso}(\Q\U)\del h|_\A={\rm id}_\A\},
$$
where $\A$ is a finite subset of $\Q\U$, while, on the other hand, the basic open neighbourhoods of the identity in ${\rm Iso}(\U)$ are of the form
$$
\{h\in {\rm Iso}(\U)\del \a x\in \A\; d(hx,x)<\eps\},
$$
where $\A$ is a finite subset of $\U$ and $\eps>0$.

In \cite{solecki} S. Solecki proved, building on work of B. Herwig and D. Lascar \cite{herwig}, the following result.
\begin{thm}[S. Solecki \cite{solecki}]\label{solecki}
Let $\A$ be a finite rational metric space. Then there is a finite rational metric space $\B$ containing $\A$ and such that any partial isometry of $\A$ extends to a full isometry of $\B$.
\end{thm}
This is turn has the consequence that ${\rm Iso}(\Q\U)$ has a comeagre conjugacy class and we can therefore refer to its elements as {\em generic} isometries of $\Q\U$. The second aim of our paper is to prove that these are all conjugate to their non-zero powers, which again suffices to show that they  all are part of an action of the additive group $(\Q,+)$ by isometries of $\Q\U$. Again, by extra care in this construction, we extend this action to an action of the locally compact ring $\aaa$.

\

In the last section we briefly consider a coarse notion of conjugacy in Polish groups. We say that $f$ and $g$ belonging to a Polish group $G$ are {\em topologically similar} if for all increasing sequences $(s_n)$ we have $f^{s_n}\Lim{n\til \infty} 1$ if and only if $g^{s_n}\Lim{n\til \infty}e$. As opposed to automorphism groups of countable structures there tend not to be comeagre conjugacy classes in large connected Polish groups and we shall provide new simple proofs of this for ${\rm Aut}([0,1],\lambda)$ and ${\rm Iso}(\U)$ by showing that in fact their topological similarity classes are meagre.

\

\noindent {\bf Acknowledgement}: I would like to thank J. Melleray, S. Solecki, R. Takloo-Bighash, and S. Thomas for many helpful discussions on the topic of this paper and, in particular, I would like to thank the anonymous referee for  many insightful remarks and especially for suggesting the relevance of the group of $p$-adic numbers, which led me to the results of Sections  \ref{adic1} and \ref{adic2}.

\section{Powers of generic measure preserving homeomorphisms}

\subsection{Free amalgams of measured Boolean algebras}
We first review the notion of free amalgams of Boolean algebras, as this will be the basis for our construction later on.
Suppose $\B_1, \B_2,\ldots, \B_n$ are finite Boolean algebras
containing a common subalgebra $\A$. We define the {\em free
amalgam}
$$
\otimes_\A^l\B_l=\B_1\otimes_\A\B_2\otimes_\A\ldots\otimes_\A\B_n
$$ of $\B_1,\ldots, \B_n$ over $\A$ as follows.

By renaming, we can suppose that $\B_i\cap \B_j=\A$ for all $i\neq
j$. We then take as our atoms the set of formal products
$$
b_1\otimes\ldots\otimes b_n,
$$
where each $b_i$ is an atom in $\B_i$ and such that for some atom
$a$ of $\A$ we have $b_i\leq a$ for all $i$. Also, for simplicity,
if $c_i\in \B_i$ is not necessarily an atom, but nevertheless we
have some atom $a$ of $\A$ such that $c_i\leq a$ for all $i$, we
write
$$
c_1\otimes\ldots\otimes c_n=\bigvee\{b_1\otimes\ldots\otimes b_n\del
b_i\text{ is an atom in }\B_i \text { and } b_i\leq c_i\}.
$$
We can now embed each $\B_i$ into $\otimes_\A^l\B_l$ by defining for
each  $b\in \B_i$, minorising an atom $a\in \A$,
$$
\pi_i(b)=a\otimes\ldots\otimes a\otimes b\otimes
a\otimes\ldots\otimes a,
$$
where the $b$ appears in the $i$'th position. In particular,
$$
\pi_i(a)=a\otimes\ldots \otimes a
$$
for all atoms $a$ of $\A$. Thus, for each $i$, $\pi_i\colon \B_i\inj
\otimes_\A^l\B_l$ is an embedding of Boolean algebras and if
$\iota_i\colon\A\inj \B_i$ denotes the inclusion mapping, then the
following diagram commutes
$$
\begin{CD}
\A @>\iota_i>> \B_i\\
@V\iota_jVV   @VV\pi_iV\\
\B_j@>>\pi_j> \otimes_\A^l\B_l
\end{CD}
$$

When $\A$ is the trivial subalgebra $\{0,1\}$, we shall write $\B_1\otimes\B_2\otimes\ldots\otimes\B_n$ instead of $\B_1\otimes_\A\B_2\otimes_\A\ldots\otimes_\A\B_n$.

Now, if $\mu_i$ are  measures on $\B_i$ agreeing on $\A$, then we
can define a new measure $\mu$ on $\otimes_\A^l\B_l$ by setting for
all $b_i\in \B_i$, minorising the same atom $a\in \A$,
$$
\mu(b_1\otimes\ldots\otimes
b_n)=\frac{\mu_1(b_1)\cdots\mu_n(b_n)}{\mu_1(a)^{n-1}}.
$$
Thus,
\begin{align*}
\mu(\pi_i(b))&=\mu(a\otimes\ldots\otimes a\otimes b\otimes a\otimes
\ldots\otimes
a)\\
&=\frac{\mu_1(a)\cdots\mu_{i-1}(a)\mu_i(b)\mu_{i+1}(a)\cdots\mu_n(a)}{\mu_1(a)^{n-1}}\\
&=\frac{\mu_1(a)\cdots\mu_{1}(a)\mu_i(b)\mu_{1}(a)\cdots\mu_1(a)}{\mu_1(a)^{n-1}}\\
&=\mu_i(b).
\end{align*}
So $\pi_i\colon (\B_i,\mu_i)\til (\otimes_\A^l\B_l,\mu)$ is an
embedding of measured Boolean algebras.

A special case is when $\A$ and each $\B_i$ are {\em equidistributed
dyadic} algebras, i.e., have $2^k$ atoms  each of measure $2^{-k}$
for some $k\geq 0$. Then this implies that for each $i$, all atoms
of $\A$ are the join of the same number of atoms of $\B_i$, namely,
$2^{k_i-m}$, where $\A$ has $2^m$ atoms and $\B_i$ has $2^{k_i}$
atoms. In this case, one can verify that $\otimes_\A^l\B_l$ has
$2^{k_1+\ldots+k_n-(n-1)m}$ atoms each of measure
$2^{(n-1)m-k_1-\ldots-k_n}$. So again this is an equidistributed
dyadic algebra.

A similar construction works for {\em equidistributed } algebras,
i.e., those having a finite number of atoms of the same (necessarily
rational) measure. In this case, the amalgam is also
equidistributed.

There is, of course, a well known graphical representation of the amalgamated product of two Boolean algebras, which is useful for guiding the intuition. For example, consider an amalgam of two measured Boolean algebras $\B$ and $\C$ over a common subalgebra $\A$ with atoms $a_1,\ldots,a_4$ and where we have made explicit the atoms of $\B\otimes_\A\C$ below $a_1\otimes a_1$:
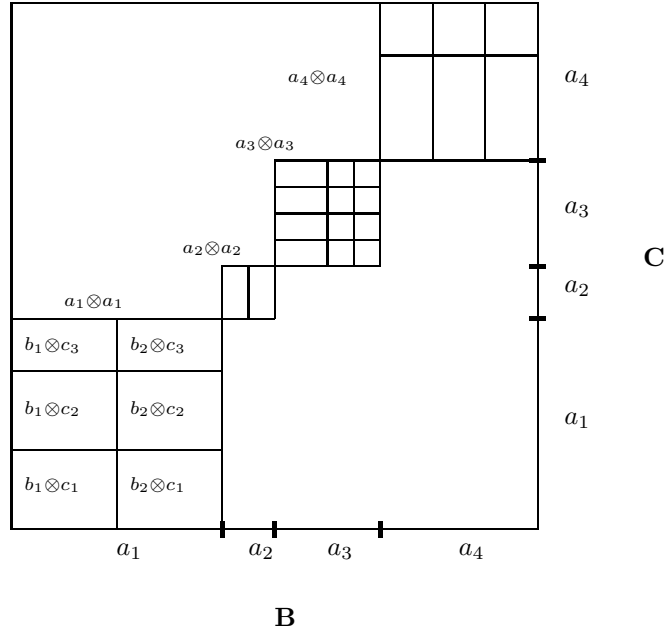
\setlength{\unitlength}{.35cm}
\begin{center}
\begin{picture}(22,21)

\put(0,0){\line(1,0){20}}
\put(0,0){\line(0,1){20}}
\put(20,0){\line(0,1){20}}
\put(0,20){\line(1,0){20}}

\put(8,0){\line(0,1){10}}
\put(10,8){\line(0,1){6}}
\put(14,10){\line(0,1){10}}

\put(0,8){\line(1,0){10}}
\put(8,10){\line(1,0){6}}
\put(10,14){\line(1,0){10}}

\put(10,-3.7){$\B$}
\put(24,10){$\C$}

\put(2,8.5){$\scriptstyle a_1\otimes a_1$}
\put(6.5,10.5){$\scriptstyle a_2\otimes a_2$}
\put(8.5,14.5){$\scriptstyle a_3\otimes a_3$}
\put(10.5,17){$\scriptstyle a_4\otimes a_4$}

\put(4,-1){$a_1$}
\put(9,-1){$a_2$}
\put(12,-1){$a_3$}
\put(17,-1){$a_4$}

\put(21,4){$a_1$}
\put(21,9){$a_2$}
\put(21,12){$a_3$}
\put(21,17){$a_4$}

\put(4,0){\line(0,1){8}}
\put(0,3){\line(1,0){8}}
\put(0,6){\line(1,0){8}}

\put(12,10){\line(0,1){4}}
\put(10,11){\line(1,0){4}}
\put(10,12){\line(1,0){4}}
\put(10,13){\line(1,0){4}}
\put(16,14){\line(0,1){6}}
\put(18,14){\line(0,1){6}}
\put(14,18){\line(1,0){6}}
\put(9,8){\line(0,1){2}}
\put(13,10){\line(0,1){4}}

\put(0.5,1.5){$\scriptstyle b_1\otimes c_1$}
\put(4.5,1.5){$\scriptstyle b_2\otimes c_1$}

\put(0.5,4.4){$\scriptstyle b_1\otimes c_2$}

\put(4.5,4.4){$\scriptstyle b_2\otimes c_2$}

\put(0.5,6.8){$\scriptstyle b_1\otimes c_3$}

\put(4.5,6.8){$\scriptstyle b_2\otimes c_3$}

\linethickness{1.5pt}
\put(20.3,8){\line(-1,0){.6}}
\put(20.3, 10){\line(-1,0){.6}}
\put(20.3,14){\line(-1,0){.6}}

\put(8,-0.3){\line(0,1){.6}}
\put(10,-.3){\line(0,1){.6}}
\put(14,-.3){\line(0,1){.6}}

\end{picture}
\end{center}

\

\

\

\

In general, an automorphism of a finite Boolean algebra arises from
a permutation of the atoms, but in the case of equidistributed, resp.
dyadic equidistributed, algebras, any permutation of the atoms conversely gives rise to a
measure preserving automorphism. Thus, for equidistributed algebras
an automorphism is necessarily a measure preserving automorphism and
we can therefore be a bit forgetful about the measure.

Suppose $\A$ is an equidistributed Boolean algebra. By a {\em partial automorphism} of $\A$ we understand an isomorphism $\phi\colon \B\til \C$ between two subalgebras $\B$ and $\C$ of $\A$ preserving the measure.

\begin{lemme}
Let $\A$ be an equidistributed, resp. dyadic equidistributed, finite Boolean algebra. Then
any measure preserving partial automorphism of $\A$ extends to an
automorphism of $\A$.
\end{lemme}

\begin{proof}
Suppose that $\B$ and $\C$ are subalgebras of $\A$ and $g\colon
\B\til \C$ a measure preserving isomorphism. Then if $b$ is an atom
of $\B$, we have, as $g$ is measure preserving, that $b$ and $g(b)$
are composed of the same number of atoms of $\A$. Therefore, we can
extend $g$ to an automorphism of $\A$ by choosing a bijection
between the constituents of $b$ and $g(b)$ for each atom $b$ of
$\B$.
\end{proof}

\subsection{Roots of measure preserving homeomorphisms}

\begin{prop}\label{divisibility}
Suppose $\A\subseteq \B$ are equidistributed, resp. dyadic equidistributed, Boolean
algebras, $g$ an automorphism of $\A$ and $f$ an automorphism of
$\B$ such that $f|\A=g^n$. Then there is an equidistributed, resp. dyadic equidistributed,
algebra $\C\supseteq \B$ and an automorphism $h$ of $\C$ extending
$g$ and such that $h^n|\B=f$.
\end{prop}

\begin{proof}
Enumerate the atoms of $\A$ as $a_1,\ldots, a_m$ and the atoms of
$\B$ as
$$
b_1^1,b_1^2,\ldots, b_1^k,b_2^1,b_2^2,\ldots, b_2^k,\ldots,
b_m^1,b_m^2,\ldots,b_m^k,
$$
where
$$
a_i=b_i^1\vee b_i^2\vee \ldots \vee b_i^k.
$$
Since $g$ is an automorphism of $\A$ we can find a permutation
$\phi$ of $\{1,\ldots,m\}$ such that
$$
g(a_i)=a_{\phi(i)}
$$
for all $i$. Similarly, we can find a function $\psi\colon
\{1,\ldots,m\}\times \{1,\ldots, k\}\til \{1,\ldots,k\}$ such that
for all $i$ and $j$
$$
f(b_i^j)=b_{\phi^n(i)}^{\psi(i,j)}.
$$
For $f(a_i)=g^n(a_i)=a_{\phi^n(i)}$ and thus $f(b_i^j)\leq f(a_i)=
a_{\phi^n(i)}$, whence $f(b_i^j)=b_{\phi^n(i)}^{\psi(i,j)}$ for some
$\psi(i,j)\in\{1,\ldots, k\}$. Also, since
\begin{align*}
b_{\phi^n(i)}^{\psi(i,1)}\vee b_{\phi^n(i)}^{\psi(i,2)}
\vee\ldots\vee b_{\phi^n(i)}^{\psi(i,k)}
&=f(b_i^1\vee b_i^2\vee \ldots \vee b_i^k)\\
&=f(a_i)\\
&=a_{\phi^n(i)}\\
&=b_{\phi^n(i)}^1\vee b_{\phi^n(i)}^2\vee \ldots \vee
b_{\phi^n(i)}^k,
\end{align*}
we see that $\psi(i,\cdot)\colon \{1,\ldots,k\}\til \{1,\ldots,k\}$
is a bijection for each $i$.

Let $\B_1=\B_2=\ldots=\B_n=\B$ and consider the free amalgam
$\otimes_\A^l\B_l$. We can now define the automorphism $h$ of
$\otimes_\A^l\B_l$ as follows.
\begin{align*}
h(b_{i}^{j_1}\otimes b_{i}^{j_2}\otimes\ldots\otimes b_{i}^{j_n})
&=b_{\phi(i)}^{\psi(i,j_n)}\otimes b_{\phi(i)}^{j_1}\otimes\ldots
\otimes b_{\phi(i)}^{j_{n-1}}.
\end{align*}
It follows from the fact that $\psi(i,\cdot)$ is a bijection that
$h$ also is a bijection of the atoms of $\otimes_\A^l\B_l$ and thus
defines an automorphism of $\otimes_\A^l\B_l$. Consider now
\begin{align*}
h^n(b_{i}^{j_1}\otimes b_{i}^{j_2}\otimes\ldots\otimes b_{i}^{j_n})
&=h^{n-1}(b_{\phi(i)}^{\psi(i,j_n)}\otimes
b_{\phi(i)}^{j_1}\otimes\ldots \otimes b_{\phi(i)}^{j_{n-1}})\\
&=h^{n-2}(b_{\phi^2(i)}^{\psi(\phi(i),j_{n-1})}\otimes
b_{\phi^2(i)}^{\psi(i,j_n)}\otimes b_{\phi^2(i)}^{j_1}\otimes\ldots
\otimes b_{\phi^2(i)}^{j_{n-2}})\\
&=\ldots\\
&=b_{\phi^n(i)}^{\psi(\phi^{n-1}(i),j_{1})}\otimes
b_{\phi^n(i)}^{\psi(\phi^{n-2}(i),j_2)}\otimes \ldots \otimes
b_{\phi^n(i)}^{\psi(i,j_n)}.
\end{align*}
Thus,
\begin{align*}
h^n(a_i\otimes a_i\otimes\ldots&\otimes a_i\otimes b_{i}^{j_n})\\
=&h^n\big(\bigvee_{j_1=1}^k \bigvee_{j_2=1}^k \ldots
\bigvee_{j_{n-1}=1}^k b_{i}^{j_1}\otimes
b_{i}^{j_2}\otimes\ldots\otimes
b_{i}^{j_n}\big) \\
=&\bigvee_{j_1=1}^k \bigvee_{j_2=1}^k \ldots \bigvee_{j_{n-1}=1}^k
h^n\big(b_{i}^{j_1}\otimes
b_{i}^{j_2}\otimes\ldots\otimes b_{i}^{j_n}\big)\\
=&\bigvee_{j_1=1}^k \bigvee_{j_2=1}^k \ldots \bigvee_{j_{n-1}=1}^k
b_{\phi^n(i)}^{\psi(\phi^{n-1}(i),j_1)}\otimes
b_{\phi^n(i)}^{\psi(\phi^{n-2}(i),j_2)}\otimes \ldots \otimes
b_{\phi^n(i)}^{\psi(i,j_n)}\\
=&a_{\phi^n(i)}\otimes a_{\phi^n(i)}\otimes \ldots \otimes
a_{\phi^n(i)}\otimes b_{\phi^n(i)}^{\psi(i,j_n)}.
\end{align*}
Similarly,
\begin{align*}
h(a_i\otimes a_i\otimes\ldots\otimes a_i)
=&h\big(\bigvee_{j_1=1}^k \bigvee_{j_2=1}^k \ldots
\bigvee_{j_{n}=1}^k b_{i}^{j_1}\otimes
b_{i}^{j_2}\otimes\ldots\otimes
b_{i}^{j_n}\big) \\
=&\bigvee_{j_1=1}^k \bigvee_{j_2=1}^k \ldots \bigvee_{j_{n}=1}^k
h\big(b_{i}^{j_1}\otimes
b_{i}^{j_2}\otimes\ldots\otimes b_{i}^{j_n}\big)\\
=&\bigvee_{j_1=1}^k \bigvee_{j_2=1}^k \ldots \bigvee_{j_{n}=1}^k
b_{\phi(i)}^{\psi(i,j_n)}\otimes b_{\phi(i)}^{j_1}\otimes\ldots
\otimes b_{\phi(i)}^{j_{n-1}}\\
=&a_{\phi(i)}\otimes a_{\phi(i)}\otimes \ldots \otimes a_{\phi(i)}.
\end{align*}

We now identify $\B$ with the image of $\B_n$ by the embedding
$\pi_n$ of $\B_n$ into $\otimes_\A^l\B_l$. Thus, the atoms of $\B$
are of the form
$$
a_i\otimes a_i\otimes\ldots\otimes a_i\otimes b_{i}^{j}
$$
and the atoms of $\A$ are
$$
a_i\otimes a_i\otimes\ldots\otimes a_i.
$$
Moreover, $g$ acts by
\begin{align*}
g(a_i\otimes a_i\otimes\ldots\otimes a_i) &=g(a_i)\otimes
g(a_i)\otimes\ldots \otimes g(a_i)\\
&=a_{\phi(i)}\otimes a_{\phi(i)}\otimes \ldots\otimes a_{\phi(i)},
\end{align*}
while $f$ acts by
\begin{align*}
f(a_i\otimes a_i\otimes\ldots\otimes a_i\otimes b_{i}^{j})
&=a_{\phi^n(i)}\otimes a_{\phi^n(i)}\otimes\ldots\otimes
a_{\phi^n(i)}\otimes b_{\phi^n(i)}^{\psi(i,j)}.
\end{align*}
Therefore, $h$ extends $g$, while $h^n$ extends $f$, which was what
we wanted.
\end{proof}

\begin{prop}\label{power}
Let $n\geq 1$. Then the generic measure preserving homeomorphism of
Cantor space is conjugate with its $n$th power.
\end{prop}

We recall that by a theorem of Kechris and the author \cite{turbulence}, there is a comeagre conjugacy class $C$ in ${\rm Homeo}(2^\N, \mu)$ and thus it makes sense to speak of the elements of this conjugacy class as the generic elements of ${\rm Homeo}(2^\N, \mu)$.

Also, note that the  basic open sets in ${\rm Homeo}(2^\N,\mu)$ are of
the form
$$
U(h,\A)=\{g\in {\rm Homeo}(\ca,\mu)\del g|_\A=h|_\A\},
$$
where $\A$ is a finite equidistributed subalgebra of $\B_\infty$ and
$h\in {\rm Homeo}(\ca,\mu)$. We shall use this notation throughout.

\begin{proof}We claim that for any $U(h,\A)$ there
is some finite equidistributed $\B\subseteq \B_\infty$ containing
$\A$ and some measure preserving homeomorphism $k$ leaving $\B$
invariant, such that $U(k,\B)\subseteq U(h,\A)$. To see this, suppose  $h$ and
$\A$ are given. Then for some $n$, both $\A$ and $h(\A)$ are subalgebras of the equidistributed algebra $\B$ having atoms $N_s=\{x\in 2^\N\del s\sqsubseteq x\}$, where $s\in 2^n$. By equidistribution, the partial automorphism $h\colon \A\til
h(\A)$ of $\B$ extends to an automorphism $\hat h$ of $\B$. So let
$k$ be any measure preserving homeomorphism of $\ca$ that extends
$\hat h$. Then $\B$ is $k$-invariant while $U(k,\B)\subseteq
U(h,\A)$.

For simplicity, if $k$ is an automorphism of a finite equidistributed algebra $\B$, we also write $U(k,\B)$ to denote the set $\{g\in {\rm Homeo}(\ca,\mu)\del g|_\B=k\}$. The previous claim amounts to the fact that the open sets $U(k,\B)$, where $k$ is an automorphism of a finite equidistributed algebra $\B$, form a $\pi$-basis for the topology, i.e., any open set contains some such $U(k,\B)$. However, they do not form a basis, as for example, a Bernoulli shift has no non-trivial finite invariant subalgebras and therefore does not belong to any $U(k,\B)$ of this form for $\B\neq \{\tom,2^\N\}$.

Let now $C$ be the comeagre conjugacy class of ${\rm
Homeo}(\ca,\mu)$  and find dense open sets $V_i\subseteq {\rm
Homeo}(\ca,\mu)$ such that $C=\bigcap_iV_i$.  Enumerate the clopen
subsets of $\ca$ as $a_0,a_1,a_2,\ldots$. We shall define a sequence
of finite equidistributed algebras $\A_0\subseteq \A_1\subseteq
\A_2\subseteq \ldots$ of clopen sets and automorphisms $g_i$ and
$f_i$ of $\A_i$ such that
\begin{itemize}
  \item[(1)] $a_i\in \A_{i+1}$,
  \item[(2)] $g_{i+1}$ extends $g_i$,
  \item[(3)] $f_{i+1}$ extends $f_i$,
  \item[(4)] $g_i^n=f_i$,
  \item[(5)] $U(g_{i+1},\A_{i+1})\subseteq V_{i}$,
  \item[(6)] $U(f_{i+1},\A_{i+1})\subseteq V_{i}$.
\end{itemize}
To begin, let $\A_0$ be the trivial algebra with automorphism
$g_0=f_0$. So suppose $\A_i$, $g_i$, and $f_i$ are defined. We let
$\B$ be an equidistributed algebra containing both $a_i$  and $\A_i$
and let $h$ be any automorphism of $\B$ extending $g_i$. As $V_i$ is
dense open we can find some $U(k,\C)\subseteq V_i$, where $\C$ is a
$k$-invariant equidistributed algebra containing $\B$ and $k$
extends $h$. Again, as $V_i$ is dense open, we can  find some $U(p,
\D)\subseteq V_i$, where $\D$ is a equidistributed algebra
containing $\C$, $p$ a measure preserving homeomorphism leaving $\D$
invariant and extending $k^n|_\C$.

Now, by Proposition \ref{divisibility}, we can find an
equidistributed algebra $\E$ containing $\D$ and an automorphism $q$
of $\E$  extending $k|_\C$ such that $q^n$ extends $p|_\D$. Finally,
set $\A_{i+1}=\E$,
$$
g_{i+1}=q\supseteq k|_\C\supseteq h\supseteq g_i,
$$
and
$$
f_{i+1}=q^n\supseteq p|_\D\supseteq k^n|_\C\supseteq g_i^n=f_i.
$$
Then $U(g_{i+1},\A_{i+1})\subseteq U(k,\C)\subseteq V_n$ and
$U(f_{i+1},\A_{i+1})\subseteq U(p,\D)\subseteq V_n$.

Set now $g=\bigcup_ig_i$ and $f=\bigcup_if_i$. By (1),(2), and (3),
$f$ and $g$ are measure preserving automorphisms of $\B_\infty$ and,
thus by Stone duality, measure preserving homeomorphisms of $\ca$.
And by (4), $g^n=f$, while by (5) and (6), $f, g\in \bigcap_iV_i=C$.
Thus, $f$ and $g$ belong to the comeagre conjugacy class and are
therefore mutually conjugate.
\end{proof}

\begin{prop}\label{inverse}
Let $G$ be  a Polish group with a comeagre conjugacy class. Then the
generic element of $G$ is conjugate to its inverse.
\end{prop}

\begin{proof}
Let $C$ be the comeagre conjugacy class of $G$. Then also $C\inv$ is
comeagre, so must intersect $C$ in some point $g$. Thus both $g$ and
$g\inv$ are generic and hence conjugate.  Now, being conjugate with
your inverse is a conjugacy invariant property and thus holds
generically in $G$.
\end{proof}

\begin{thm}\label{homeomorphisms}
Let $n\neq 0$. Then the generic measure preserving homeomorphism of
Cantor space is conjugate with its $n$'th power and hence has roots
of all orders.

Thus, for the generic measure preserving homeomorphism $g$, there is
an action of $(\Q,+)$ by measure preserving homeomorphisms of $\ca$
such that $g$ is the action by $1\in \Q$.
\end{thm}

\begin{proof} By Propositions \ref{power} and \ref{inverse}, we know that the generic $g$ is conjugate to all its positive powers
and to $g\inv$. But then $g\inv$ is generic and thus conjugate to
$(g\inv)^n=g^{-n}$, whence $g$ is conjugate with $g^{-n}$, $n\geq
1$.

So suppose $g$ is generic and $n\geq 1$. Then,  there is some $f$
such that $(fgf\inv )^n=fg^nf\inv =g$, and hence $g$ has a generic
$n$th root, namely, $fgf\inv$. This means that we can define a
sequence $g=g_1,g_2,\ldots$ of generic elements such that $g_{n+1}$
is an $(n+1)$st root of $g_n$, $(g_{n+1})^{n+1}=g_n$. The following
therefore defines an embedding of $(\Q,+)$ into ${\rm
Homeo}(\ca,\mu)$ with $1=\frac1{1!}\mapsto g_1$,
$$
\frac{k}{n!}\mapsto g_{n}^k,
$$
$k\in \Z$, $n\geq 1$.
\end{proof}

\subsection{The ring of finite ad\`eles}\label{loc comp}
Fix a prime number $p$ (the reader is referred to the article ``Global Fields'' by J. W. S. Cassels in \cite{cassels} for more details of the following construction). We recall the $p$-adic valuation on $\Q$, which is the function $|\cdot|_p\colon \Q\til [0,+\infty[$ defined by $|0|_p=0$ and
$$
\Big|p^k\frac ab\Big|_p=p^{-k},
$$
whenever $a,b$ are non-zero integers not divisible by $p$ and $k\in \Z$. It is easily seen that $|st|_p=|s|_p\cdot |t|_p$ and $|s+t|_p\leq \max\{|s|_p,|t|_p\}$ for all $s,t\in \Q$. It follows that $d_p(s,t)=|s-t|_p$ defines a translation invariant metric on $\Q$ such that if $(s_n)$ and $(t_n)$ are Cauchy sequences in $\Q$ then so are $(s_n\inv)$, $(s_n+t_n)$, and $(s_nt_n)$. Thus, if $\Q_p$ denotes the metric completion of  $\Q$, $\Q_p$ is a topological field, known as the field of {\em $p$-adic numbers}.

One way of representing the elements of the field $\Q_p$ is as infinite series
$$
\sum_{i=k}^\infty a_ip^i,
$$
where $k\in \Z$, $a_i\in \{0,\ldots, p-1\}$. Note that any such series is $d_p$-Cauchy. Moreover, the valuation extends to all of $\Q_p$ by
$$
\Big|\sum_{i=k}^\infty a_ip^i\Big|_p=p^{-k},
$$
assuming $a_k\neq 0$. Here, the usual ring of integers $\Z$ can be recognised as the set of finite series $\sum_{i=0}^ka_ip^i$, where $k<\infty$. The closure of $\Z$ within $\Q_p$, called the ring of {\em $p$-adic integers} and denoted by $\Z_p$, consists of all expressions $\sum_{i=0}^\infty a_ip^i$ and is a compact, open subgroup
subring of $\Q_p$. Note however that
$$
\Z_p=\{x\in \Q_p\del |x|_p\leq 1\},
$$
so despite its name, $\Z_p$ contains all rational numbers of the form $p^k\frac ab$, where $k\geq 0$ and $a,b$ are non-zero integers not divisible by $p$.

We now define the restricted product $\prod_{p}'\Q_p$ with respect to the compact open subsets $\Z_p$. $\prod_p'\Q_p$ consists of all sequences $(s_p)\in \prod_p\Q_p$, where the index $p$ runs over all primes, such that $s_p\in \Z_p$ for all but finitely many primes $p$. Moreover, $\prod_p'\Q_p$ has as basis for its topology the sets of the form
$$
\prod_{p\in F}U_p\times \prod_{p\notin F}\Z_p,
$$
where  $F$ is a finite set of primes and $U_p\subseteq \Q_p$ is open for all $p\in F$. In particular, we see that $\prod_p\Z_p$ is a compact open subring of $\prod_p'\Q_p$.

Now if $s\in\Q^*$, then writing
$$
s=\frac{p_1^{n_1}\ldots p_l^{n_l}}{q_1^{m_1}\ldots q_k^{m_k}},
$$
where $p_i$ and $q_i$ are distinct primes and $n_i,m_i\in \N$, we see that $|s|_p=1$ for all $p\neq p_i,q_i$, and so if $s_p$ denotes the element of $\Q_p$ corresponding to $s$, then
$(s_p)\in \prod_p'\Q_p$. It follows that we can identify $\Q$ with a subfield of the ring $\prod_p'\Q_p$ via the embedding $s\mapsto (s_p)$. Also, if $s\in \Q$ is such that $(s_p)\in\prod_p\Z_p$, then $|s|_p\leq 1$ for all $p$, so actually $s\in \Z$. Therefore, if $(t_n)$ is a sequence in $\Q$, we see that $t_n\til 0$ in the $\prod_p'\Q_p$-topology if and only if $t_n\in \Z$ for all but finitely many $n$ and, moreover, for any power $p^k$ of a prime, $k\geq 1$, $t_n$ is an integer multiple of $p^k$ for all but finitely many $n$.

The ring $\prod_p'\Q_p$ is called the {\em ring of finite ad\`eles} and will henceforth be denoted by $\aaa$. A fact, which will be important to us, is that $\Q$ is a dense subset of $\aaa$. This follows from the Strong Approximation Theorem (Cassels \cite{cassels}, \S 15). Also, $\aaa$ is a locally compact ring.

\

We shall now present another direct construction of $\aaa$, which is  closer to the viewpoint of this article (one can consult the book by L. Ribes and P. Zalesskii \cite{ribes} for more information on the profinite completion of $\Z$). Consider the embedding $\theta$ of $\Z$ into the group $\prod_{n=1}^\infty\Z/{n\Z}$ given by $\theta(a)=(a(n))_{n=1}^\infty$, where $a(n)\equiv a \mod n$ for every $n$. We define the {\em profinite completion} $\Z$ to be compact subgroup of $\prod_{n=1}^\infty\Z/{n\Z}$ given by
$$
\hat \Z=\overline {\theta(\Z)},
$$
and see that $\hat \Z$ is the subgroup consisting of all sequences $(a(n))_{n\geq 1}$ such that $a(m)\equiv a(n) \mod n$, whenever $n$ divides $m$.
Identifying $\Z$ with its image by $\theta$, the induced topology is called the {\em profinite topology} on $\Z$.  So if $(a_i)$ is a sequence in $\Z$, then $a_i\til 0$ in the profinite topology if and only if for every integer $n$, $a_i\equiv 0\mod n$ for all but finitely many $i$.
It follows from the Chinese Remainder Theorem  that  $\hat \Z\iso \prod_p\Z_p$.

Now, let $\|\cdot\|$ be the norm on $\Q$ defined by setting $\|0\|=0$ and for $s\in \Q^*$
$$
\|s\|=2^{-\min (n\geq 1\,|\, \frac sn\notin \Z)}.
$$
Then for any $n\geq 1$ and $s,t\in \Q$,  if $\frac sn, \frac tn\in \Z$, also $\frac{s+t}n,\frac {st}n\in \Z$, which implies that
$$
\|s+t\|\leq \max \{\|s\|,\|t\|\}
$$
 and
 $$
 \|st\|\leq \max \{\|s\|,\|t\|\}.
 $$
We can now define a translation invariant ultra-metric on $\Q$ by $d(s,t)=\|s-t\|$, and notice that from $\|s+t\|\leq \max \{\|s\|,\|t\|\}$ it follows that if $(s_n)$ and $(t_n)$ are Cauchy sequences, then so is $(s_n+t_n)$.

For $s\in \Q^*$, we define a clopen subgroup of $\Q$ by
$$
\langle s\rangle=\{ ns\;\Big|\; n\in \Z\}.
$$
To see that it is open, just note that if $s=\frac ab$, with $a,b\in \Z$, then
$$
(*)\qquad (\langle s\rangle)_{2^{-a}}=\{t\in \Q\del d(t,\langle s\rangle)<2^{-a}\}=\langle s\rangle.
$$
For if $\|t- ns\|<2^{-a}$, where $t\in \Q$ and $n\in \Z$, then $t-ns=la$ for some $l\in \Z$ and so
$t=(lb+n)\frac ab\in\langle s\rangle$. It follows that $\langle s\rangle$ is also closed, since the complement $\Q\setminus \langle s\rangle$ is the union of its disjoint open cosets. Remark that $\Q$ is the increasing union of the clopen subgroups $\langle \frac 1{n!}\rangle$ and that the $d$-topology on $\Z=\langle 1\rangle$ coincides with the profinite topology.

We claim that if $(s_n)$ and $(t_n)$ are Cauchy sequences in $\Q$, then so is $(s_nt_n)$. To see this, note that by $(*)$ the $s_n$ and $t_n$ will eventually all belong to some common subgroup $\langle 1/k\rangle$ and hence can be written $s_n=\frac{a_n}k$ and $t_n=\frac{b_n}k$ for integers $a_n,b_n$. Then, if $d\geq 1$ is fixed, for all sufficiently large $n,m$,
$$
\frac {s_n-s_m}{kd}=\frac{a_n-a_m}{k^2d}\in \Z
$$
and
$$
\frac {t_n-t_m}{kd}=\frac{b_n-b_m}{k^2d}\in \Z,
$$
so
\begin{displaymath}\begin{split}
\frac {s_nt_n-s_mt_m}d&=\frac{a_nb_n-a_mb_m}{k^2d}\\
&=\frac{(a_n-a_m)(b_n-b_m)}{k^2d}+\frac{(a_n-a_m)b_n}{k^2d}+\frac{a_n(b_n-b_m)}{k^2d}\in \Z.
\end{split}\end{displaymath}
Since $d$ is arbitrary it follows that $\|s_nt_n-s_mt_m\|\Lim{n,m\til \infty}0$, so ($s_nt_n)$ is Cauchy.

Thus, we see that the operations of addition $+$ and multiplication $\cdot$ on $\Q$ extend to continuous ring operations $+$ and $\cdot$ on the  $d$-metric completion of $\Q$. Since $\Z=\langle 1\rangle$ is open in $\Q$, $\hat \Z$ is a compact open subgroup of the completion, so the completion is locally compact.  We note also that a sequence $t_n\in \Q$ converges to $0$ if and only if for all natural numbers $k$, $t_n$ is an integer multiple of $k$ for all but finitely many $n$, i.e., if and only if $t_n\til 0$ in the $\prod_p'\Q_p$-topology. So the $d$-metric completion of $\Q$ is topologically isomorphic to $\aaa$.

\subsection{Actions of $\aaa$ by measure preserving homeomorphisms}\label{adic1}

Now returning to generic elements of ${\rm Homeo}(2^\N,\mu)$, we note that if $g$ is generic, then every $g$-orbit on the algebra $\B_\infty$ of clopen sets is finite. This follows from the fact, established in the proof of Proposition \ref{power}, that the open sets $U(k,\B)$, where $k$ is an automorphism of a finite equidistributed algebra $\B$, form a $\pi$-basis for the topology. So given any $b\in \B_\infty$, the generic $g$ must belong to some such $U(k,\B)$, where $\B$ is an equidistributed algebra containing $b$, and thus the $g$-orbit of $b$ is contained in the finite $g$-invariant algebra $\B$. Using this, one sees that $\overline{\langle g\rangle}$ is a profinite subgroup of ${\rm Homeo}(2^\N,\mu)$. Conversely, any generic $g$ has orbits of any finite order.

Therefore, if $k_i\in \Z$, we have
$$
(*)\qquad g^{k_i}\Lim{i\til \infty}e \quad\equi\quad \a n \; \a^\infty i  \; \;\;\;k_i\equiv 0\mod n
$$
and so the mapping $k\in \Z\mapsto g^k\in {\rm Homeo}(2^\N,\mu)$ is a topological isomorphism between $\Z$ equipped with its profinite topology and the infinite cyclic topological subgroup $\langle g\rangle$ of ${\rm Homeo}(2^\N,\mu)$. By the completeness of ${\rm Homeo}(2^\N,\mu)$, this extends to a topological embedding of $\hat \Z$ into ${\rm Homeo}(2^\N,\mu)$ whose image is the closed subgroup $\overline{\langle g\rangle}$. We shall now see how to extend this embedding to $\aaa$.

\begin{thm}\label{locally compact completion}
Let $g$ be a generic element of ${\rm Homeo}(2^\N,\mu)$ and let $\aaa$ be the ring of finite ad\`eles. Then $1\in \Q\mapsto g\in{\rm Homeo}(2^\N,\mu)$ extends to a homeomorphic embedding of $(\aaa,+)$ into ${\rm Homeo}(2^\N,\mu)$.
\end{thm}
We note that, since the ring $\aaa$ contains $\Q$ as a subfield, $(\aaa,+)$ is a divisible group. So, by the above theorem, the generic measure preserving homeomorphism lies in a divisible, locally compact, Abelian subgroup of ${\rm Homeo}(2^\N,\mu)$. 

The proof is done by carefully choosing the sequence $g=g_1, g_2, g_3,\ldots$ of generic elements in the proof of Theorem \ref{homeomorphisms}, so as to control the convergence of sequences $(g_{n_i})^{k_i}$.  We split the proof into a couple of lemmas.

\begin{lemme}\label{caf}
Let $\B\subseteq \B_\infty$ be an equidistributed dyadic subalgebra and suppose $g,h$ are generic elements of ${\rm Homeo}(2^\N,\mu)$ with $g[\B]=h[\B]=\B$ and $g|_\B=h|_\B$. Then $g$ and $h$ are conjugate by an element of ${\rm Homeo}(2^\N,\mu)_\B=U(e,\B)$.
\end{lemme}

\begin{proof}
Note that $U(g,\B)=U(h,\B)$ is an open subset of ${\rm Homeo}(2^\N,\mu)$ that is invariant under the conjugacy action by ${\rm Homeo}(2^\N,\mu)_\B$. Also, as $g$ and $h$ are generic, it follows from Proposition 3.2 of \cite{turbulence} that $X=\{fgf\inv \del f\in {\rm Homeo}(2^\N,\mu)_\B\}$ and   $Y=\{fhf\inv \del f\in {\rm Homeo}(2^\N,\mu)_\B\}$ are comeagre in neighbourhoods $V$ of $g$ , resp. $W$ of $h$. So find equidistributed dyadic algebras $\C\supseteq \B$ and $\D\supseteq \B$ that are respectively $g$ and $h$-invariant such that $U(g,\C)\subseteq V$ and $U(h,\D)\subseteq W$. It suffices to show that for some $f\in {\rm Homeo}(2^\N,\mu)_\B$, we have $f\inv U(g,\C)f \cap U(h,\D)\neq \tom$ since then also $f\inv Xf \cap Y\neq \tom$.

Now let $\C\otimes_\B\D$ be the free amalgam of $\C$ and $\D$ over $\B$ and define a measure preserving automorphism $k\colon \C\otimes_\B\D\til \C\otimes_\B\D$ by setting
$$
k(c\otimes d)=g(c)\otimes h(d),
$$
whenever $c\in \C$ and $d\in \D$ minorise some common atom $b\in \B$. Note that in this case, since $g[\B]=h[\B]=\B$ and $g|_\B=h|_\B$, we have $g(c),h(d)\leq g(b)=h(b)$, showing that the image $g(c)\otimes h(d)$ is well defined as an element of $\C\otimes_\B\D$.

Now, embedding the subalgebra $\C\otimes_\B\B$ of  $\C\otimes_\B\D$ into $\B_\infty$ via
$$
c\otimes b\mapsto c
$$
and subsequently extending this embedding to an embedding $\iota$ of all of $\C\otimes_\B\D$ into $\B_\infty$, we see that $k'=\iota\circ k\circ \iota\inv $ is an automorphism of $\iota[\C\otimes_\B\D]$ such that $k'|_\C=g|_\C$. Extend now $k'$ arbitrarily to a measure preserving homeomorphism of $2^\N$ also denoted by $k'$.

We can now find a measure preserving homeomorphism $f\in {\rm Homeo}(2^\N,\mu)$ such that $f(d)=\iota(b\otimes d)$ for all $d\in \D$ minorising an atom $b$ of $\B$. Note that then $f|_\B={\rm id}_\B$, so $f\in {\rm Homeo}(2^\N,\mu)_\B$. Also,
\begin{displaymath}\begin{split}
&f\inv k' f(d)=f\inv  k'(\iota(b\otimes d))=f\inv\iota k \iota\inv (\iota(b\otimes d))\\
=&f\inv \iota k(b\otimes d)=f\inv \iota(g(b)\otimes h(d)) =h(d),
\end{split}\end{displaymath}
whenever $d\in \D$ minorises an atom $b$ of $\B$. So $f\inv k'f|_\D=h|_\D$. Thus, $k'\in U(g,\C)$ while $f\inv k'f\in U(h,\D)$, so $f\inv U(g,\C)f\cap U(h,\D)\neq \tom$, which finishes the proof.
\end{proof}

\begin{lemme}\label{fine root}
Suppose $\B$ is a dyadic, equidistributed Boolean algebra, $g$ an automorphism of $\B$ and $b\in \B\setminus \{0,1\}$ is an element having $g$-period $k$, i.e.,
$g^i(b)=b$ if and only if $k$ divides $i$. Then for any $n\geq 1$ there is a dyadic, equidistributed Boolean algebra $\C\supseteq \B$ and an automorphism $h$ of $\C$ such that $h^n|_\B=g$ and $b$ has $h$-period $kn$.
\end{lemme}

\begin{proof}
Let ${\C}=\underbrace{\B\otimes \cdots\otimes\B}_{n \text{ times}}$ and identify $\B$ with the last factor in the product, i.e., $x\in \B$ corresponds to $1\otimes\cdots\otimes 1\otimes x$. Now, if $x_1\otimes\cdots\otimes x_{n-1}\otimes x_n$ is an atom of $\C$, we define
$$
h(x_1\otimes\cdots\otimes x_{n-1}\otimes x_n)=g(x_n)\otimes x_1\otimes\cdots\otimes x_{n-1}.
$$
Then clearly
\begin{displaymath}\begin{split}
h^n(1\otimes\cdots\otimes 1\otimes x_n)&=g(1)\otimes\cdots\otimes g(1)\otimes g(x_n)\\
&=1\otimes\cdots\otimes 1\otimes g(x_n),
\end{split}\end{displaymath}
so $h^n|_\B=g$.

Also, for any $l\in \Z$,
\begin{displaymath}\begin{split}
h^{nl}(1\otimes\cdots\otimes 1\otimes b)&=1\otimes\cdots\otimes 1\otimes g^l(b),
\end{split}\end{displaymath}
which equals $1\otimes\cdots\otimes 1\otimes b$ if and only if $k$ divides $l$.
And if $n$ does not divide $m$, then $\pi_n(h^m(1\otimes\cdots\otimes 1\otimes b))=1\neq b$, where $\pi_n$ is the projection onto the last coordinate factor. So $h^m(1\otimes\cdots\otimes 1\otimes b)\neq 1\otimes\cdots\otimes 1\otimes b$. Thus, $1\otimes\cdots\otimes 1\otimes b$ has $h$-period $nk$.
\end{proof}

\begin{lemme}\label{period}
Suppose $g\in {\rm Homeo}(2^\N,\mu)$ is generic and $b\in \B_\infty$ has $g$-period $k$. Then for any $n\geq 1$ there is a generic $f$ such that $g=f^n$ and $b$ has $f$-period $kn$.
\end{lemme}

\begin{proof}
Let $\B$ be the minimal $g$-invariant equidistributed subalgebra of $\B_\infty$ containing $b$.  Now, by Lemma \ref{fine root}, there is an equidistributed algebra $\C\supseteq \B$ and an automorphism $\tilde h$ of $\C$ such that $\tilde h|_\B=g|_\B$, while $b$ has $\tilde h$-period $kn$. Let now $h\in {\rm Homeo}(2^\N,\mu)$ be any generic extension of $\tilde h$. Then $h^n|\B=g|_\B$ and $h^n$ is generic too, by Proposition \ref{power}. Applying Lemma \ref{caf}, $h^n$ and $g$ are conjugacte by an element $s\in {\rm Homeo}(2^\N,\mu)_\B$, $sh^ns\inv=g$, whereby $f=shs\inv$ is a generic $n$th root of $g$ with respect to which $b$ has period $kn$.
\end{proof}

And now for the proof of Theorem \ref{locally compact completion}.
\begin{proof}
Suppose $g$ is generic and let $b\in \B_\infty\setminus \{\tom,2^\N\}$ be an arbitrary clopen set fixed by $g$.  Using Lemma \ref{period}, we can inductively choose generic $g=g_1,g_2,g_3,\ldots$ such that $(g_{n+1})^{n+1}=g_n$ and $b$ has $g_n$-period $n!$. It follows from looking at $(g_{n_i})^{k_i}(b)$, that if $k_i, n_i\geq 1$ are such that $(g_{n_i})^{k_i}\Lim{i\til \infty}e$, then for all but finitely many $i$, $n_i!$ divides $k_i$ and, in particular, for all but finitely many $i$, $(g_{n_i})^{k_i}$ is an integer power of $g=g_1$, namely
$$
(g_{n_i})^{k_i}=g^{\frac{k_i}{n_i!}}.
$$
So using $(*)$, we see that
\begin{displaymath}\begin{split}
(g_{n_i})^{k_i}\Lim{i\til \infty}e\quad&\equi\quad \big[ \a^\infty i\;\;\; k_i\equiv 0\mod n_i!\;\;\og\;\;  \a m\; \a^\infty i\;\;\; \frac{k_i}{n_i!}\equiv 0\mod m\big]\\
&\equi\quad\Big\|\frac{k_i}{n_i!}\Big\|\til 0.
\end{split}\end{displaymath}
Now, embedding $\Q$ into ${\rm Homeo}(2^\N,\mu)$ by $\frac k{n!}\mapsto g_n^k$, $k\in \Z$ and $n\geq 1$, and identifying $\Q$ with its image, we see that the topology induced from ${\rm Homeo}(2^\N,\mu)$ coincides with the $d$-topology.  Therefore, by completeness of ${\rm Homeo}(2^\N,\mu)$, the embedding extends to all of $\aaa$.
\end{proof}

It is instructive to see how the multiplication in the topological ring $\aaa$ is interpreted as a multiplication operation $\times$ in the topological group ${\rm Homeo}(2^\N,\mu)$ via the embedding above. So suppose $g$ is  generic and let $g=g_1,g_2,g_3,\ldots$ be the sequence of generic roots defined in the proof of Theorem \ref{locally compact completion}. Denote by $\go R(g)$ the image of $\aaa$ under the embedding into ${\rm Homeo}(2^\N,\mu)$. We shall simplify notation a bit by letting $g^s$ denote the image of $s\in \aaa$ in $\go R(g)$. Thus $g^sg^t=g^{s+t}$ and $g^s\times g^t=g^{st}$, whenever $s,t\in \aaa$.
In particular, to compute the square root with respect to the group multiplication of an element $g^s$ of $\go R(g)$, we multiply by $g^{\frac1 2}$:
$$
(g^{\frac 12}\times g^s)(g^{\frac 12}\times g^s)=g^{\frac s2}g^{\frac s2}=g^s.
$$

The rational powers of $g$ are of course easy to write in terms of integer powers of $g_n$, namely, if $s=\frac k{n!}$ for $k\in \Z$ and $n\geq 1$, then $g^s=(g_n)^k$. On the other hand, if $s,t\in \aaa$ are arbitrary elements, we find some $k$ such that $s,t\in \overline{\langle \frac 1k\rangle}$, whereby $st\in \overline{\langle\frac 1{k^2}\rangle}$.
Now, given any $b\in \B_\infty$, let $\ku O$ be the orbit of $b$ under $g^{k^{-2}}$.
Then if $m{k^{-1}}, n{k^{-1}}\in \langle \frac 1{k}\rangle$ are sufficiently close to $s$ and $t$ respectively, $g^{m{k^{-1}}}$ and $g^{n{k^{-1}}}$ agree with $g^s$ and $g^t$ on $\ku O$. In particular,
$$
[g^s\times g^t](b)=[g^{m{k^{-1}}}\times g^{n{k^{-1}}}](b)=(g^{k^{-2}})^{mn}(b).
$$

\section{Powers of generic isometries}

\subsection{Free amalgams of metric spaces}
We shall now review the concept of free amalgamations of metric spaces, which is certainly part of the folklore.  Suppose $\A$ and $\B_1,\ldots,\B_n$ are non-empty finite metric spaces
and $\iota_i\colon {\A}\inj\B_i$ is an isometric embedding for each $i$.  We define the {\em free amalgam}
$\bigsqcup_{\A}\B_l$ of $\B_1,\ldots,\B_n$ over $\A$ and the embeddings $\iota_1,\ldots,\iota_n$ as follows.

Denote by $d_i$ the metric on $\B_i$ for each $i$ and let $\C_i=\B_i\setminus \iota_i[\A]$. By renaming elements, we can suppose that $\C_1,\ldots,\C_n$ and $\A$ are pairwise disjoint.

We then let the universe of $\bigsqcup_{\A}\B_l$ be
$\A\cup \bigcup_{i=1}^n\C_i$ and define the metric $\partial$ by the
following conditions
\begin{itemize}
\item[(1)] $\partial(x,y)=d_i(\iota_ix,\iota_iy)$ for $x,y\in \A$,
\item[(2)] $\partial(x,y)=d_i(\iota_ix,y)$ for $x\in \A$ and $y\in \C_i$,
\item[(3)] $\partial(x,y)=d_i(x,y)$ for $x,y\in \C_i$,
\item[(4)] $\partial(x,y)=\min_{z\in\A}d_i(x,\iota_iz)+d_j(\iota_jz,y)$ for $x\in \C_i$ and $y\in \C_j$, $i\neq j$.
\end{itemize}
We notice first that in (1) the definition is independent of $i$ since each $\iota_i$ is an isometry. Also, a careful checking of the triangle inequality shows that this indeed defines a metric $\partial$ on $\A\cup \bigcup_{i=1}^n\C_i$.

We define for each $i$ an isometric embedding $\pi_i\colon \B_i\inj \bigsqcup_{\A}\B_l$ by
\begin{itemize}
\item $\pi_i(x)=x$ for $x\in \C_i$,
\item $\pi_i(\iota_ix)=x$ for $x\in \A$.
\end{itemize}
Notice that in this way the following diagram commutes
$$
\begin{CD}
\A @>\iota_i>> \B_i\\
@V\iota_jVV   @VV\pi_iV\\
\B_j@>>\pi_j> \bigsqcup_\A\B_l
\end{CD}
$$

\subsection{Roots of isometries}

\begin{prop}\label{divisibility isometry}
Let $\bf A\subseteq B$ be finite rational metric spaces, $f$ an isometry of $\A$ and $g$ an isometry of $\B$ leaving $\A$ invariant and such that $f^n=g|_\A$ for some $n\geq 1$. Then there is a finite rational metric
space $\bf D\supseteq B$ and an isometry $h$ of $\D$ such that $h^n$ leaves $\B$ invariant and  $h^n|_\B=g$.
\end{prop}

\begin{proof}
Let $\B_1=\ldots=\B_{n}=\B$ and define isometric embeddings $\iota_i\colon \A\inj\B_i$ by
$$
\iota_i(x)=f^{-i}(x).
$$
To distinguish between the different copies of $\B$, we let for $x\in \B\setminus \A$, $x^i$ denote the copy of $x$ in $\C_i=\B_i\setminus \iota_i[\A]=\B_i\setminus\A$. Note also that $\B=\B_1=\ldots=\B_n$ all have the same metric, which we denote by $d$.
We now define $h$ on $\bigsqcup_\A\B_l$ as follows.
\begin{itemize}
\item $h(x)=f(x)$ for $x\in \A$,
\item $h(x^i)=x^{i+1}$ for $x\in \B\setminus \A$ and $1\leq i<n$,
\item $h(x^n)=(gx)^1$ for $x\in \B\setminus \A$.
\end{itemize}
Now, obviously, $h$ is a permutation of $\A$ and for $1\leq i<n$, $h$ is a bijection between $\C_i$ and $\C_{i+1}$. Moreover, $h$ is a bijection between $\C_n$ and $\C_1$. Therefore, $h$ is a permutation of $\bigsqcup_\A\B_l$. We check that $h$ is $1$-Lipschitz.

Suppose first that $x,y\in \A$. Then
\begin{align*}
\partial(hx,hy)&=\partial(fx,fy)\\
&=d(\iota_ifx,\iota_ify)\\
&=d(f^{-i}fx,f^{-i}fy)\\
&=d(f^{1-i}x,f^{1-i}y)\\
&=d(f^{-i}x,f^{-i}y)\\
&=d(\iota_ix,\iota_iy)\\
&=\partial(x,y).
\end{align*}
Also, $h$ is clearly an isometry between $\C_i$ and $\C_{i+1}$ for $1\leq i< n$. So consider the case $\C_n$. Fix $x,y\in \B\setminus \A$. Then
\begin{align*}
\partial(h(x^n),h(y^n))&=\partial((gx)^1,(gy)^1)\\
&=d(gx,gy)\\
&=d(x,y)\\
&=\partial(x^n,y^n).
\end{align*}
Now, if $x\in \A$, $y\in \B\setminus \A$, and $1\leq i<n$, then
\begin{align*}
\partial(h(x),h(y^i))&=\partial(fx,y^{i+1})\\
&=d(\iota_{i+1}fx,y)\\
&=d(f^{-(i+1)} fx,y)\\
&=d(f^{-i}x,y)\\
&=d(\iota_ix,y)\\
&=\partial(x,y^i).
\end{align*}
Also, if $x\in \A$, $y\in \B\setminus \A$, then
\begin{align*}
\partial(h(x),h(y^n))&=\partial(fx,(gy)^1)\\
&=d(\iota_{1}fx,gy)\\
&=d(f^{-1}fx,gy)\\
&=d(x,gy)\\
&=d(g\inv x,y)\\
&=d(f^{-n}x,y)\\
&=d(\iota_nx,y)\\
&=\partial(x,y^n).
\end{align*}
And finally, if $x,y\in \B\setminus \A$ and $1\leq i< j\leq n$, we pick $z\in \A$ such that the distance $\partial(x^i,y^j)$ is witnessed by $z$, i.e.,
\begin{align*}
\partial(x^i,y^j)=d(x,\iota_iz)+d(\iota_jz,y)=d(x,f^{-i}z)+d(f^{-j}z,y).
\end{align*}
Assume first that $j<n$. Then
\begin{align*}
\partial(h(x^i),h(y^j))&=\partial(x^{i+1},y^{j+1})\\
&\leq d(x,\iota_{i+1}fz)+d(\iota_{j+1} fz,y)\\
&= d(x,f^{-i}z)+d(f^{-j}z,y)\\
&=\partial(x^i,y^j).
\end{align*}
And if $j=n$, we have
\begin{align*}
\partial(h(x^i),h(y^n))&=\partial(x^{i+1},(gy)^1)\\
&\leq d(x,\iota_{i+1}fz)+d(\iota_{1} fz,gy)\\
&= d(x,f^{-i}z)+d(z,f^ny)\\
&= d(x,f^{-i}z)+d(f^{-n}z,y)\\
&=\partial(x^i,y^n).
\end{align*}
Thus, $h$ is an isometry of $\bigsqcup_\A\B_l$.

Now see $g$ and $f$ as isometries of the first copy $\B_1$ of $\B$, i.e., $g(x^1)=(gx)^1$ for $x^1\in \C_1$. Let $\pi_1\colon \B_1\inj \bigsqcup_{\A}\B_l$ be the canonical isometric embedding defined by
\begin{itemize}
\item $\pi_1(x^1)=x^1$ for $x^1\in \C_1$,
\item $\pi_1(\iota_1x)=x$ for $x\in \A$.
\end{itemize}
To finish the proof, we need to show that the following diagram commutes
$$
\begin{CD}
\B_1 @>g>> \B_1\\
@V\pi_1VV   @VV\pi_1V\\
\bigsqcup_{\A}\B_l@>>h^n> \bigsqcup_{\A}\B_l
\end{CD}
$$
First, suppose $y=\iota_1x\in \A$. Then
\begin{align*}
h^n\pi_1y=&h^n\pi_1\iota_1x=h^nx=f^nx\\
=&\pi_1\iota_1f^nx
=\pi_1f\inv f^nx
=\pi_1f^nf\inv x\\
=&\pi_1f^n\iota_1x
=\pi_1f^ny
=\pi_1gy.
\end{align*}
Now suppose that $x\in \B\setminus \A$. Then
\begin{align*}
h^n\pi_1(x^1)=&h^n(x^1)=h(x^n)=(gx)^1
=\pi_1(gx)^1=\pi_1g(x^1).
\end{align*}
\end{proof}

\begin{prop}\label{power isometry}
Let $n\geq 1$. Then the generic isometry of
the rational Urysohn metric space is conjugate with its $n$th power.
\end{prop}

Again, the reference to the generic isometry of the rational Urysohn metric space is justified by the existence of a comeagre conjugacy class in its isometry group, a fact established by Solecki in \cite{solecki}.

\begin{proof}A basic open set in ${\rm Iso}(\Q\U)$ is of the form
$$
U(h,\A)=\{g\in {\rm Iso}(\Q\U)\del g|_\A=h|_\A\},
$$
where $\A$ is a finite subspace  of $\Q\U$ and
$h\in {\rm Iso}(\Q\U)$. We claim that for any $U(h,\A)$ there
is some finite $\B\subseteq \Q\U$ containing
$\A$ and some isometry $k$ leaving $\B$
invariant, such that $U(k,\B)\subseteq U(h,\A)$.
For if $h$ and $\A$ are given, choose by Theorem \ref{solecki} some finite $\B\subseteq \Q\U$ containing both $\A$
and $h(\A)$ such that the partial isometry $h\colon \A\til h(\A)$ of $\A\cup h(\A)$ extends to an isometry $\hat h$ of $\B$. Let
$k$ be any isometry of $\Q\U$ that extends $\hat h$. Then $\B$ is $k$-invariant while $U(k,\B)\subseteq
U(h,\A)$.

Again, if $k$ is an isometry of some finite $\B\subseteq \Q\U$, we let $U(k,\B)=\{g\in {\rm Iso}(\Q\U)\del g|_\B=k\}$.

Let now $C$ be the comeagre conjugacy class of ${\rm Iso}(\Q\U)$  and find dense open sets $V_i\subseteq {\rm Iso}(\Q\U)$ such that $C=\bigcap_iV_i$.  Enumerate the points of $\Q\U$ as $a_0,a_1,a_2,\ldots$. We shall define a sequence
of finite subsets $\A_0\subseteq \A_1\subseteq
\A_2\subseteq \ldots\subseteq \Q\U$ and isometries $g_i$ and
$f_i$ of $\A_i$ such that
\begin{itemize}
  \item[(1)] $a_i\in \A_{i+1}$,
  \item[(2)] $g_{i+1}$ extends $g_i$,
  \item[(3)] $f_{i+1}$ extends $f_i$,
  \item[(4)] $g_i^n=f_i$,
  \item[(5)] $U(g_{i+1},\A_{i+1})\subseteq V_{i}$,
  \item[(6)] $U(f_{i+1},\A_{i+1})\subseteq V_{i}$.
\end{itemize}
To begin, let $\A_0=\tom$ with trivial isometries $g_0=f_0$. So suppose $\A_i$, $g_i$, and $f_i$ are defined. We let
$\B\subseteq \Q\U$ be a finite subset containing both $a_i$  and $\A_i$
and such that there is some isometry $h$ of $\B$ extending $g_i$. As $V_i$ is
dense open we can find some $U(k,\C)\subseteq V_i$, where $\C\subseteq \Q\U$ is a
$k$-invariant finite set containing $\B$ and $k$
extends $h$. Again, as $V_i$ is dense open, we can  find some $U(p,
\D)\subseteq V_i$, where $\D\subseteq \Q\U$ is a finite set
containing $\C$, $p$ an isometry of $\Q\U$ leaving $\D$
invariant and extending $k^n|_\C$.

Now, by Proposition \ref{divisibility isometry}, we can find a finite subset $\E\subseteq \Q\U$ containing $\D$ and an isometry $q$
of $\E$  extending $k|_\C$ such that $q^n$ extends $p|_\D$. Finally,
set $\A_{i+1}=\E$,
$$
g_{i+1}=q\supseteq k|_\C\supseteq h\supseteq g_i,
$$
and
$$
f_{i+1}=q^n\supseteq p|_\D\supseteq k^n|_\C\supseteq g_i^n=f_i.
$$
Then $U(g_{i+1},\A_{i+1})\subseteq U(k,\C)\subseteq V_n$ and
$U(f_{i+1},\A_{i+1})\subseteq U(p,\D)\subseteq V_n$.

Set now $g=\bigcup_ig_i$ and $f=\bigcup_if_i$. By (1),(2), and (3),
$f$ and $g$ are isometries of $\Q\U$.
And by (4), $g^n=f$, while by (5) and (6), $f, g\in \bigcap_iV_i=C$.
Thus, $f$ and $g$ belong to the comeagre conjugacy class and are
therefore mutually conjugate.
\end{proof}

Now in exactly the same way as for measure preserving homeomorphisms, we can prove

\begin{thm}\label{isometries}
Let $n\neq 0$. Then the generic isometry of
the rational Urysohn metric space is conjugate with its $n$'th power and hence has roots
of all orders.

Thus, for the generic isometry $g$, there is
an action of $(\Q,+)$ by isometries of $\Q\U$
such that $g$ is the action by $1\in \Q$.
\end{thm}

\subsection{Actions of $\aaa$ by isometries on $\Q\U$.}\label{adic2}
In a similar manner as for measure preserving homeomorphisms, it is now possible to show that any generic isometry extends to an action of the ring $\aaa$.

\begin{thm}\label{locally compact completion for isometries}
Let $g$ be a generic element of ${\rm Iso}(\Q\U)$. Then $1\in \Q\mapsto g\in{\rm Iso}(\Q\U)$ extends to a homeomorphic embedding of $(\aaa,+)$ into ${\rm Iso}(\Q\U)$.
\end{thm}

Since this is done almost exactly as for measure preserving homeomorphisms, modulo replacing dyadic, equidistributed Boolean algebras with finite metric spaces, we shall not overextend our claims to the reader's attention and instead just give the exact statements of the needed lemmas.

\begin{lemme}
Let $\B\subseteq \Q\U$ be a finite subset and suppose $g,h$ are generic elements of ${\rm Iso}(\Q\U)$ with $g[\B]=h[\B]=\B$ and $g|_\B=h|_\B$. Then $g$ and $h$ are conjugate by an element of ${\rm Iso}(\Q\U)_\B=U(e,\B)$.
\end{lemme}

\begin{lemme}
Suppose $\B$ is a finite rational metric space, $g$ an isometry of $\B$ and $b\in \B$ is a point having $g$-period $k$, i.e.,
$g^i(b)=b$ if and only if $k$ divides $i$. Then for any $n\geq 1$ there is a finite rational metric space $\C\supseteq \B$ and an isometry $h$ of $\C$ such that $h^n|_\B=g$ and $b$ has $h$-period $kn$.
\end{lemme}

\begin{lemme}
Suppose $g\in {\rm Iso}(\Q\U)$ is generic and $b\in \Q\U$ has $g$-period $k$. Then for any $n\geq 1$ there is a generic $f$ such that $g=f^n$ and $b$ has $f$-period $kn$.
\end{lemme}

\section{Topological similarity and Rohlin's Lemma for isometries}

Suppose $G$ is a Polish group and $f,g\in G$. We say that $f$ and $g$ are {\em
topologically similar} if the topological groups $\langle f\rangle\leq G$ and
$\langle g\rangle\leq G$ are isomorphic. We should note here that $\langle f\rangle$ refers to the cyclic group generated by $f$ and not its closure. By the completeness of Polish groups, if $\langle f\rangle$ and $\langle g\rangle$ are isomorphic, then so are $\overline{\langle f\rangle}$ and $\overline{\langle g\rangle}$, but not vice versa (for an example of this, one can consider irrational rotations of the circle).

Notice first that any $f$ is topologically similar to $f\inv$. For if
$\psi(f^n)=f^{-n}$, then $\psi$ is an involution homeomorphism, since inversion
is continuous in $G$. Of course, if $f$ and $g$ have infinite order, then any
isomorphic homeomorphism $\phi$ between $\langle f\rangle$ and $\langle
g\rangle$ must send the generators to the generators and so either $\phi(f)=g$
or $\phi(f)=g\inv$. But then composing with $\psi$ we can always suppose that
$\phi(f)=g$.

Moreover, to see that $\phi\colon f^n\mapsto g^n$ is a topological group isomorphism between $\langle f\rangle$ and $\langle g\rangle$, it is enough to check continuity at the identity $e$ of both $\phi$ and $\phi\inv$. But, letting $\{U_i\}_{i\in \N}$ be an open neighbourhood basis at $e$ in $G$, this clearly holds if and only if
$$
(*)\quad \a i\; \e j\; \a n\; \big[(f^n\in U_j\til g^n\in U_i)\;\&\;(g^n\in U_j\til f^n\in U_i)\big].
$$
Notice also that as  $\langle f\rangle$ and $\langle g\rangle$ are metrisable, $f$ and $g$ are
topologically similar if and only if for all increasing sequences
$(s_n)\subseteq \N$, $f^{s_n}\Lim{n}e$ if and only if $g^{s_n}\Lim{n}e$.
By $(*)$, topological similarity is a Borel equivalence relation. Actually, it is ${\bf \Pi}^0_3$, which can be seen by noting that $(*)$ is equivalent to
$$
\a i\; \e j\; \a n\; \big[(f^n\notin U_j\;\vee\; g^n\in \overline{U_i})\;\&\;(g^n\notin U_j\;\vee\; f^n\in \overline{U_i})\big].
$$
We
notice also that topological similarity is really independent of the ambient
group $G$. For example, if $G$ is topologically embedded into another Polish
group $H$, then $f$ and $g$ are topologically similar in $G$ if and only if
they are topologically similar in $H$.

Topological similarity is an obvious invariant for conjugacy, that is, if there
is any way to make $f$ and $g$ conjugate in some Polish group, then they have
to be topologically similar.

Of particular interest are the cases $G={\rm Aut}([0,1],\lambda)$,
$G=U(\ell_2)$, and $G={\rm Iso}(\U)$. We recall that the group ${\rm Aut}([0,1],\lambda)$ of Lebesgue measure preserving automorphisms of the unit interval is equipped with the so called {\em weak topology}: It is the weakest topology such that for all Borel sets $A,B\subseteq [0,1]$ the map $g\mapsto\lambda(gA\triangle B)$ is continuous.
Also, ${\rm Aut}([0,1],\lambda)$
sits inside of $U(\ell_2)$ via the Koopman representation and two measure
preserving transformations $f$ and $g$ are said to be {\em spectrally
equivalent} if they are conjugate in $U(\ell_2)$. By the spectral theorem,
spectral equivalence is Borel. Also, topological similarity is strictly coarser than
spectral equivalence. To see this, we notice that mixing is not a topological
similarity invariant, whereas it is a spectral invariant. For if $f$ is mixing,
then the automorphism $f\oplus {\rm id}$ is a non-mixing transformation of
$[0,1]\oplus[0,1]$ but generates a discrete subgroup of ${\rm Aut}([0,1]\oplus
[0,1],\lambda\oplus \lambda)$. So taking a transformation $h\in {\rm
Aut}([0,1],\lambda)$ conjugate with $f\oplus {\rm id}$, we see that $f$ and $h$
are topologically similar, since they both generate discrete groups.
A survey
of the closely related topic of topological torsion elements in topological
groups is given by Dikranjan in \cite{dikranjan}.

\begin{prop}Let $G$ be a non-trivial Polish group  such that for all infinite $S\subseteq \N$ and neighbourhoods $V\ni e$ the set
$\mathbb A(S,V)=\{g\in G\del \e s\in S\; g^s\in V\}$ is dense. Then every topological similarity class of $G$ is meagre.

Moreover, for every infinite  $S\subseteq \N$ the set
$$
\mathbb C(S)=\{g\in G\del \e (s_n)\subseteq S\;\;  g^{s_n}\Lim{n} e\}
$$
is dense $G_\delta$ and invariant under topological similarity.
\end{prop}

\begin{proof}
Let $V_0\supseteq V_1\supseteq\ldots$ be a basis of open neighbourhoods of the identity and notice that for any infinite $S\subseteq \N$,
\begin{align*}
\mathbb C(S)&=\{g\in G\del \e (s_n)\subseteq S\;\;  g^{s_n}\Lim{n}e\}\\
&=\{g\in G\del \a k\; \e n\in S\setminus [1,k]\;\; g^n\in V_k\}\\
&=\bigcap_k\mathbb A(S\setminus [1,k],V_k).
\end{align*}
Moreover, as every $A(S\setminus [1,k],V_k)$ is open and dense, $\mathbb C(S)$ is dense $G_\delta$ and invariant under topological similarity.

Now, if some topological similarity class $C$ was nonmeagre, then
$$
C\subseteq \bigcap_{\substack{S\subseteq \N\\\text{ infinite}}}\mathbb C(S)
$$
and hence for all $g\in C$, $g^n\Lim{n}e$, implying that $g=e$, which is impossible.
\end{proof}

Since by Rohlin's Lemma the sets $\{g\in {\rm Aut}([0,1],\lambda)\del \e s\in
S\; g^s=e\}$ are dense in ${\rm Aut}([0,1],\lambda)$ for all infinite
$S\subseteq \N$, we have
\begin{cor}
Every topological similarity class is meagre in ${\rm Aut}([0,1],\lambda)$.
\end{cor}

This improves a result variously attributed to Rohlin or del Junco \cite{junco} saying that all conjugacy classes are meagre
in ${\rm Aut}([0,1],\lambda)$. We clearly see the importance of Rohlin's Lemma
in these matters. However, interestingly, Rohlin's Lemma can also be used to
prove the existence of {\em dense} conjugacy classes in ${\rm
Aut}([0,1],\lambda)$.

It is of interest to note that the same argument applies to the unitary group $U(\ell_2)$ (see Chapter 1.2 in Kechris' book \cite{book}). Thus every topological similarity class in $U(\ell_2)$ is meagre. Moreover, $U(\ell_2)$ embeds into ${\rm Aut}([0,1],\lambda)$ via the Gaussian measure construction. So in this case the conjugacy classes in $U(\ell_2)$ induced by ${\rm Aut}([0,1],\lambda)$ still remain meagre.

We now have the following analogue of Rohlin's Lemma for isometries of the Urysohn metric space.
\begin{prop}[Rohlin's Lemma for isometries]\label{rohlin}
Suppose $S\subseteq \N$ is infinite. Then the set
$$
\{g\in {\rm Iso}(\U)\del \e n\in S\; g^n=e\}
$$
is dense in ${\rm Iso}(\U)$.
\end{prop}

Similar sounding statements can certainly be found in the literature, for example, it follows easily from Lemma 5.3.7. in Pestov's book \cite{pestov} that the set of isometries of finite order is dense in ${\rm Iso}(\U)$, but the quantitative statement above, i.e., depending on $S\subseteq \N$, does not seem to follow easily from the more abstract methods of \cite{pestov}. We therefore include the simple proof of Proposition \ref{rohlin} below.

A {\em finite cyclic order} is a finite subset $\F$ of the unit circle $S^1$.
If $x\in \F$, we denote by $x^+$ the first $y\in \F$ encountered by moving
counterclockwise around $S^1$ beginning at $x$. We then denote $x$ by $y^-$,
i.e., $x^+=y$ if and only if $y^-=x$.

\begin{lemme}
Suppose $h$ is an isometry of $\U$ and $\delta>0$. Then for all finite
$\A\subseteq \U$ there is an isometry $f$ of $\U$ such that $d(f(a),h(a))\leq
\delta$ for all $a\in \A$ while $d(a,f(b))\geq \delta$ for all $a,b\in \A$.
\end{lemme}

\begin{proof}
Let $\B=\A\cup h[\A]$ and let $\C=\B\times \{0,\delta\}$ be equipped with the
$\ell_1$-metric $d_1((b,x),(b',y))=d(b,b')+|x-y|$. Clearly, $\B$ is isometric
with $\B\times \{0\}$ and $\B\times \{\delta\}$, so we can assume that $\B$ is
actually $\B\times \{0\}\subseteq \C\subseteq \U$. Now, let $f$ be any isometry
of $\U$ such that $f(a,0)=(h(a),\delta)$ for $a\in \A$.
\end{proof}

Now for the proof of Proposition \ref{rohlin}.
\begin{proof}
Suppose $\A\subseteq \U$ is finite, $h$ an isometry of $\U$, and $\eps>0$. We wish to find some isometry $g$ such that $d(g(a),h(a))<\eps$ for all $a\in \A$ and such that for some $s\in S$, $g^s=e$. Find first some $f$ such that $d(f(a),h(a))<\eps$ for all $a\in \A$ while $d(a,f(b))>\eps/2$ for all $a,b\in \A$. It is therefore enough to find some $g$ that agrees with $f$ on $\A$ while $g^s=e$ for some $s\in S$.

We let $\Delta={\rm diam}(\A\cup f[\A])$ and $\delta=\min(d(x,f(y))\del x,y\in \A)$. Fix a  number $s\in S$ such that $\delta\cdot (s-2)\geq\Delta$ and take a finite cyclic order $\F$ of cardinality $s$. Now let
$$
\B=\{a\!\bullet\! x\del a\in \A\;\&\; x\in \F\},
$$
where $a\!\bullet\! x$ are formally new points.

A {\em path} in $\B$ is a sequence $p=(a_0\!\bullet\! x_0,a_1\!\bullet\! x_1,\ldots,a_n\!\bullet\! x_n)$ where $n\geq 1$ and such that for each $i$, $x_{i+1}$ is either $x_i^-$,  $x_i$, or $x_i^+$. We define the {\em length} of a path by
$$
\ell(p)=\sum_{i=0}^{n-1}\rho(a_i\!\bullet\! x_i,a_{i+1}\!\bullet\! x_{i+1}),
$$
where
$$
\rho(a\!\bullet\! x,b\!\bullet\! y)=\left\{
                              \begin{array}{ll}
                                d(a,b), & \hbox{if $y=x$;} \\
                                d(a,f(b)), & \hbox{if $y=x^+$;} \\
                                d(f(a),b), & \hbox{if $y=x^-$,}
                              \end{array}
                            \right.
$$
and put $|p|=n+1$.

Therefore, if $\breve{p}$ denotes the reverse path of $p$ and $p\centerdot q$ the concatenation of two paths (whenever it is defined), then $\ell(\breve{p})=\ell(p)$ and $\ell(p\centerdot q)=\ell(p)+\ell(q)$. Thus, $\ell$ is the distance function in a finite graph with weighted edges and hence the following defines a  a metric on $\B$
$$
D(a\!\bullet\! x,b\!\bullet\! y)=\inf\big(\ell(p)\del p \text{ is  a path with initial point } a\!\bullet\! x \text{ and end point } b\!\bullet\! y\big).
$$

We say that two paths are {\em equivalent} if they have the same initial point and the same end point. We also say that a path $p$ is {\em positive} if either $p=(a\!\bullet\! x,b\!\bullet\! x)$ for some $x\in \F$ or $p=(a_0\!\bullet\! x_0,a_1\!\bullet\! x_1,\ldots,a_n\!\bullet\! x_n)$, where $x_{i+1}=x_i^+$ for all $i$. Similarly, $p$ is {\em negative} if either $p=(a\!\bullet\! x,b\!\bullet\! x)$ for some $x\in \F$ or $p=(a_0\!\bullet\! x_0,a_1\!\bullet\! x_1,\ldots,a_n\!\bullet\! x_n)$, where $x_{i+1}=x_i^-$ for all $i$. So $p$ is positive if and only if $\breve{p}$ is negative. Notice also that if $p$ is positive, then $\ell(p)\geq \delta\cdot (|p|-2)$.

\begin{lemme}
For every path $p$ there is an equivalent path $q$, with $\ell(q)\leq \ell(p)$, which is either positive or negative.
\end{lemme}

\begin{proof}
If $p$ is not either positive or negative, then there is a segment of $p$ of one of the following forms
\begin{equation*}
\begin{aligned}
&(1)\quad (a\!\bullet\! x,b\!\bullet\! x,c\!\bullet\! x),\\
&(2)\quad (a\!\bullet\! x^+,b\!\bullet\! x,c\!\bullet\! x),\\
&(3)\quad (a\!\bullet\! x^-,b\!\bullet\! x,c\!\bullet\! x),\\
&(4)\quad (a\!\bullet\! x,b\!\bullet\! x^+,c\!\bullet\! x),
\end{aligned}\qquad\qquad
\begin{aligned}
&(5)\quad (a\!\bullet\! x,b\!\bullet\! x^-,c\!\bullet\! x),\\
&(6)\quad (a\!\bullet\! x,b\!\bullet\! x,c\!\bullet\! x^+),\\
&(7)\quad (a\!\bullet\! x,b\!\bullet\! x,c\!\bullet\! x^-).\\
\,
\end{aligned}
\end{equation*}
We replace these by respectively
\begin{equation*}
\begin{aligned}
&(1')\quad (a\!\bullet\! x,c\!\bullet\! x),\\
&(2')\quad (a\!\bullet\! x^+,c\!\bullet\! x),\\
&(3')\quad (a\!\bullet\! x^-,c\!\bullet\! x),\\
&(4')\quad (a\!\bullet\! x,c\!\bullet\! x),
\end{aligned}\qquad\qquad
\begin{aligned}
&(5')\quad (a\!\bullet\! x,c\!\bullet\! x),\\
&(6')\quad (a\!\bullet\! x,c\!\bullet\! x^+),\\
&(7')\quad (a\!\bullet\! x,c\!\bullet\! x^-),\\
\,
\end{aligned}
\end{equation*}
and see that by the triangle inequality for $d$ we can only decrease the value of $\ell$. For example, in case (3), we see that
\begin{align*}
 \rho(a\!\bullet\! x^-,b\!\bullet\! x)+\rho(b\!\bullet\! x,c\!\bullet\! x)=&d(a,f(b))+d(b,c)\\
 =&d(a,f(b))+d(f(b),f(c))\\
 \geq&d(a,f(c))\\
 =&\rho(a\!\bullet\! x^-,c\!\bullet\! x).
 \end{align*}
 We can then finish the proof by induction on $|p|$.
\end{proof}
We now claim that $D(a\!\bullet\! x, b\!\bullet\! x)=d(a,b)$. To see this, notice first that $D(a\!\bullet\! x, b\!\bullet\! x)\leq d(a,b)$. For the other inequality, let $p$ be an either positive or negative path from $a\!\bullet\! x$ to $b\!\bullet\! x$. By symmetry, we can suppose $p$ is positive. But then, unless $p=(a\!\bullet\! x, b\!\bullet\! x)$, we must have $|p|\geq s+1$, whence also $\ell(p)\geq \delta\cdot (|p|-2)\geq \delta\cdot (s-1)\geq\Delta\geq d(a,b)$.
A similar argument shows that $D(a\!\bullet\!  x,b\!\bullet\!  x^+)=d(a,f(b))$.

This shows that for any $x_0\in \F$, $\A\cup f[\A]$ is isometric with $\A\times \{x_0,x_0^+\}$ by the function $a\mapsto a\!\bullet\! x_0$ and $f(a)\mapsto a\!\bullet\! x_0^+$. So we can just identify $\A\cup f[\A]$  with $\A\times \{x_0,x_0^+\}$. Notice also that the following mapping $g$ is an isometry of $\B$:
$$
a\!\bullet\! x\mapsto a\!\bullet\! x^+.
$$
Moreover, it agrees with $f$ on their common domain $\A\times \{x_0\}$. Realising $\B$ as a subset of $\U$ containing $\A$, we see that $g$ acts by isometries on $\B$ with  $g^s=e$. It then follows that $g$ extends to a full isometry of $\U$ still satisfying $g^s=e$.
\end{proof}

\begin{cor}
Every topological similarity class is meagre in ${\rm Iso}(\U)$.
\end{cor}
Again this strengthens a result of Kechris \cite{glasner} saying that all conjugacy classes are meagre.

\end{document}